\documentclass[11pt,a4paper]{article}
\usepackage{a4wide}
\usepackage[utf8]{inputenc}
\usepackage[T1]{fontenc}
\usepackage{xcolor}

\usepackage{amsfonts}
\usepackage[fleqn]{amsmath}
\usepackage{amssymb}
\usepackage{amsthm}

\usepackage[numbers]{natbib}
\everymath{\displaystyle}
\allowdisplaybreaks

\newtheorem{theorem}{Theorem}

\newtheorem{lemma}[theorem]{Lemma}
\theoremstyle{definition}
\newtheorem{definition}[theorem]{Definition}

\newcommand{\Jarnik}{Jarn\'{\i}k}
\newcommand{\N}{{\mathbb{N}}}

\newcommand{\words}{{\mathbb W}}
\newcommand{\E}{{\mathcal{E}}}
\newcommand{\F}{{\mathcal{F}}}
\newcommand{\G}{{\mathcal{G}}}

\newcommand{\alphabar}{\overline{\alpha}}

\makeatletter
\newcounter{subsubparagraph}[subparagraph]
\renewcommand\thesubsubparagraph{%
  \thesubparagraph.\@arabic\c@subsubparagraph}
\newcommand\subsubparagraph{%
  \@startsection{subsubparagraph}    
    {6}                              
    {\parindent}                     
    {3.25ex \@plus 1ex \@minus .2ex} 
    {-1em}                           
    {\normalfont\normalsize\itshape}}
\newcommand\l@subsubparagraph{\@dottedtocline{6}{10em}{5em}}
\newcommand{\subsubparagraphmark}[1]{}
\makeatother

\begin{document}

\title{Irrationality Exponent, Hausdorff Dimension and Effectivization}
\author{Ver\'onica Becher\thanks{Becher’s research is supported by the University of Buenos Aires and CONICET. Becher is member of the Laboratoire International Associ\'e INFINIS, CONICET/Universidad de Buenos Aires-CNRS/Universit\'e Paris Diderot.}, Jan Reimann\thanks{Reimann was partially supported by NSF grant DMS-1201263.} and Theodore A.\ Slaman\thanks{Slaman was partially supported by NSF grant DMS-1600441.}}
  
\maketitle

%

\begin{abstract}
We generalize the classical theorem by \Jarnik\ and Besicovitch on the irrationality exponents of
real numbers and Hausdorff dimension and show that the two notions are independent.
For any real number $a$ greater than or equal to $2$ and
 any non-negative real  $b$ be less than or equal to $2/a$,
we show that there is a Cantor-like
set with Hausdorff dimension equal to $b$ such that, with respect to its uniform measure, almost
all real numbers have irrationality exponent equal to~$a$.  
We give an analogous result relating
the irrationality exponent and the effective Hausdorff dimension of individual real numbers.  
We prove that there is a Cantor-like set such that, with respect to its uniform measure, almost all elements in the
set have effective Hausdorff dimension equal to~$b$ and irrationality exponent equal to~$a$. 
In each case, we obtain the desired set as a distinguished path in a tree of Cantor sets.

\medskip
\textbf{Keywords:} Diophantine approximation, Cantor Sets, Effective Hausdorff Dimension

\textbf{Mathematics Subject Classification (2010):} 11J82,11J83,03D32
\end{abstract}

The irrationality exponent~$a$ of a real number $x$ reflects how well $x$ can be approximated by
rational numbers.  Precisely, it is the supremum of the set of real numbers~$z$ for which the
inequality
\[
0< \left| x- \frac{p}{q}\right| < \frac{1}{q^z}
\]
is satisfied by an infinite number of integer pairs $(p, q)$ with $q > 0$.  Rational numbers have
irrationality exponent equal to~$1$.  It follows from the fundamental work by~\citet{Khintchine:1924a} (see also Chapter 1 of \citep{Bug2004} for a good overview) that almost all irrational numbers (with respect
to Lebesgue measure) have irrationality exponent equal to~$2$. 

On the other hand, it follows from the theory of continued fractions that for
every $a$ greater than~$2$
or equal to infinity, there is a real number $x$ with irrationality exponent equal to~$a$.

The sets of real numbers with irrationality exponent~$a$ become smaller as $a$ increases.  This is
made precise by calculating their dimensions.  For a set of real numbers $X$, a non-negative real
numbers $s$, and a real number $\delta>0$, let
\[
H^s_\delta(X)=
\inf\left\{\sum_{j\geq 1} d_{j}^s:
\text{there  is a cover of } 
X\text{ by balls with diameter }(d_j < \delta: j\geq 1) \right\}.
\]
Note that $H^s_\delta(X)$ cannot decrease as $\delta$ goes to zero. The $s$-dimensional Hausdorff measure of $X$ is given as
\[
\lim_{\delta \to 0} H^s_\delta(X).
\]
The Hausdorff dimension of a set $X$ is the infimum of the set of non-negative reals $s$ such that
the $s$-dimensional Hausdorff measure of $X$ is~zero.  

\citet{Jarnik:1929a} and independently \citet{Bes1934} showed that the Hausdorff dimension of the set of real numbers $x$ such that $x$ has irrationality exponent greater than or equal to $a$ is $2/a$.  
\citet{Guting:1963a} proved that the Hausdorff dimension of the set of reals $x$ such that $x$ has irrationality exponent (exactly) equal to $a$ is also $2/a$. Later, \citet{Beresnevich:2001a} 
gave a new proof of Güting's result and also established that the $2/a$-dimensional
 Hausdorff measure of the set of reals $x$ with irrationality exponent $a$ is infinite. 
\citet{Bugeaud:2003a} extended these results to general approximation order functions 
in the sense of \citet{Khintchine:1924a}.

The results outlined above suggest that there is a strong tie between Hausdorff dimension and irrationality exponent, in the way irrationality exponents are metrically stratified in terms Hausdorff dimension. 
This tie becomes even more evident when replacing Hausdorff dimension 
by a \emph{pointwise} counterpart, known as \emph{effective} dimension, 
which we briefly describe.

The theory of computability defines the effective versions of the classical notions.  A~computable
function from non-negative integers to non-negative integers is one which can be effectively
calculated by some algorithm.  The definition extends to functions from one countable set to
another, by fixing enumerations of those sets.  A real number $x$ is computable if there is a
computable sequence of rational numbers $(r_j)_{j\geq 0}$ such that $|x -r_j| < 2^{-j}$
for~$j \ge 0$.  A set is computably enumerable if it is the range of a computable function with
domain the set of non-negative integers.

\citet*{CaiHar1994} considered effectively presented properties which are related to Hausdorff
dimension.  \citet*{Lut2000} formulated a definition of effective Hausdorff dimension for individual
sequences in terms of computable martingales.  \citet*{ReiSte05} reformulated the notion in terms of
computably enumerable covers, as follows.  Let $\words$ be the set of finite binary sequences
(sequences of $0$s and~$1$s), and we write $\N$ for the set of non-negative integers.  A set $X$ of real numbers has effective $s$-dimensional Hausdorff
measure zero, if there exists a computably enumerable set $C\subseteq \N\times \words$ such that for
every~$n\in \N$,
$C_n=\{w\in \words : (n,w)\in C \}$
satisfies that for every $x$ in $X$ there is a length $\ell$ such that the sequence of first $\ell$
digits in the base-2 expansion of $x$ is in $C_n$, and
\[
\sum_{w\in C_n}2^{-\text{\it length}(w)s}< 2^{-n}.  
\]
The effective Hausdorff dimension of a set $X$ of real numbers  is the infimum of the set of non-negative
reals $s$ such that the effective $s$-dimensional Hausdorff measure of $X$ is~zero.  The effective
Hausdorff dimension of an individual real number~$x$ is the effective Hausdorff dimension of the
singleton set $\{x\}$.  Note that the effective Hausdorff dimension of~$x$ can be greater than~$0$:
$x$~has effective Hausdorff dimension greater than or equal to~$s$ if for all~$t<s$, $x$~avoids
every effectively presented countable intersection of open sets (namely, every effective $G_\delta$
set) of $t$-dimensional Hausdorff measure zero.  The effective notion reflects the classical one in
that the set $\{x: \text{$x$ has effective Hausdorff dimension equal to } t \}$ has Hausdorff
dimension~$t$.

The notion of effective Hausdorff dimension can also be defined in terms of computable
approximation.  Intuitively, the Kolmogorov complexity of a finite sequence is the length of the
shortest computer program that outputs that sequence.  Precisely, consider a computable function $h$
from finite binary sequences to finite binary sequences such that the domain of $h$ is an antichain.
Define the $h$-complexity of~$\tau$ to be the length of the shortest $\sigma$ such that
$h(\sigma)=\tau$.  There is a universal computable function $u$ with the property that for every
such $h$ there is a constant $c$ such that for all $\tau$, the $h$-complexity of $\tau$ is less than
the $u$-complexity of $\tau$ plus~$c$.  Fix a universal $u$ and define the prefix-free Kolmogorov
complexity of $\tau$ to be its $u$-complexity.  In these terms, the effective Hausdorff dimension of a
real number $x$ is the infimum of the set of rationals $t$ such that there is a $c$ for which there are
infinitely many $\ell$ such that the prefix-free Kolmogorov complexity of the first $\ell$ digits in
the base-2 expansion of $x$ is less than $t\cdot\ell-c$. See \citet{Dow2010} for a thorough presentation.

Effective Hausdorff dimension was introduced by \citet{Lut2000} to add computability to the notion of
Hausdorff dimension, in the same way that the theory of algorithmic randomness adds computability to
Lebesgue measure.  But, we could also view the effective Hausdorff dimension of a real number $x$ as
a counterpart 
of its irrationality exponent.  Where the irrationality exponent of $x$ reflects
how well it can be approximated by rational numbers, the effective Hausdorff dimension of a real $x$
reflects how well it can be approximated by computable numbers.  The connection is more than an
analogy. 

Except for rational numbers all real numbers have irrationality  exponent greater than or equal  to~$2$.
This means that for each irrational number  $x$, the supremum of the set 
$\{ z: \text{ there are infinitely many rationals } p/q \text { such that  } |x - p/q| < 1/q^z\}$ is greater than or equal  to $2$.
On the other hand, most real numbers have effective Hausdorff dimension~$1$ and are algorithmically random,
which means that the initial segments of their expansions  can not be described by  concise algorithms.
Thus, for any such $x$, all rationals $p/q$ provide at most the first $2\log(q)$ digits of 
the base-$2$ expansion of~$x$  (take $p$ and $q$ integers, such that $0< p < q$, 
and describe each of them with $\log q$ digits.). 
Consequently, for a rational $p/q$  is impossible that  $|x - p/q| $ be much less than $ 1/q^2$.
It follows that  for most real numbers the irrationality exponent is just equal to~$2$.
In case $x$ is  a  Liouville number, its irrationality exponent is infinite, so
 for every $n$ there is a rational $p/q$  such that   $|x - p/q| < 1/q^n$. 
Thus,  $2\log(q)$ digits  can describe the first $n\log(q)$ bits of $x$. 
Therefore, each Liouville number has effective dimension $0$~\citep{Sta02}.
\citet{Calude:2013a} generalized this argument to show that if $x$ has irrationality exponent~$a$, then the effective Hausdorff dimension of~$x$ is less than or equal to~$2/a$.  

The precise metric stratification of irrationality exponents in terms of Hausdorff dimension has therefore an exact effective counterpart. In this note we show that the two concepts, irrationality exponent and (effective) Hausdorff dimension, are nevertheless independent. More precisely, we prove the following results. 

\begin{theorem}\label{1}
  Let $a$ be a real number greater than or equal to~$2$.  For every real number $b\in[0,2/a]$ there
  is a Cantor-like set~$E$ with Hausdorff dimension equal to~$b$ such that, 
  for the uniform measure on $E$,
  almost all real numbers have irrationality exponent equal to~$a$.
\end{theorem}

\begin{theorem}\label{2}
  Let $a$ and $b$ be real numbers such that $a\geq 2$ and $b\in[0,2/a]$.  There is a Cantor-like set
  $E$ such that, for the uniform measure on $E$, almost all real numbers in  $E$ have
  irrationality exponent equal to $a$ and effective Hausdorff dimension equal to $b$.
\end{theorem}


A classic result due to \citet{Besicovitch:1952a} ensures that, for any real number $s \geq 0$, 
any closed subset of the real number  of infinite $s$-dimensional Hausdorff measure has a subset of finite, non-zero $s$-dimensional Hausdorff measure. 
It follows that   the set of reals of irrationality exponent greater than or equal to 
$a$ has a subset of Hausdorff dimension $b$, for any $0 \leq b \leq 2/a$.
However, the proof of this theorem  \citep[which uses binary net measures, see for example][]{Fal1986}
does not preserve the Cantor set structure of \Jarnik's fractal, in particular it does not provide a nice measure concentrated on a set of reals of irrationality exponent~$a$. 
In fact, as shown by \citet{Kjos-Hanssen-Reimann:2010a}, finding a Cantor subset of finite, non-zero Hausdorff measure is generally very hard.

As we will see, it takes some effort to ensure the persistence of the Cantor set structure (and with it a nice measure) when passing 
to smaller dimensions while preserving irrationality exponents. Furthermore, the case of effective dimension presents additional 
difficulties since we have to replace real numbers (which may be non-computable), by rational approximations to them.

\subsubsection*{Acknowledgments}
The authors thank two anonymous referees for helpful comments and remarks.

\section{\Jarnik's Fractal and Its Variations}

Let $a$ be a real number greater than~2.  As mentioned earlier, \citet{Jar1928} and \citet{Bes1934}
independently established that the set of real numbers with irrationality exponent greater than or
equal to~$a$ has Hausdorff dimension~$2/a$.  \Jarnik\ exhibited a Cantor-like set $E^J(a)$ such that
every element of $E^J(a)$ has irrationality exponent greater than or equal to~$a$ and such that, for
every $d$ greater than~$a$, the set of real numbers with irrationality exponent~$d$ is null for the uniform measure~$\mu_J$
on $E^J(a)$.  The latter condition followed by application of the mass distribution principle
on~$E^J(a)$:

\begin{lemma}[Mass Distribution Principle, cf.~\citealp{Fal2003}]\label{3}
  Let $\mu$ be a measure on $E,$ a subset of the real numbers, and let $a$ be a positive real number.  If $\mu(E) >0$
  and there are positive constants $c$ and $\delta$ such that for every interval $I$ with
  $|I|<\delta$, $\mu(I) < c |I|^a$, then the Hausdorff dimension of $E$ is greater than or equal
  to~$a$.
\end{lemma}

For $b$ given so that $0\leq b<2/a$, we will define a subset $E$ of $E^J(a)$ with dimension~$b$.
$E$~will also be a Cantor-like set  and have its own uniform measure $\mu$.  
Using $\mu_J$ as a guide, we will ensure that for any $d$ greater than $a$
the set of real numbers with irrationality exponent~$d$ is null with respect to~$\mu$.  Further, we
shall arrange that $\mu$ has the mass distribution property for exponent~$b$.

We fix some notation to be used in the final step of the construction.  For a positive
integer $p$,
\[
G_{p}(a)=\left\{x\in \left(\frac{1}{p^{a}}, 1-\frac{1}{p^{a}}\right): \exists q\in\N,
  \left|\frac{q}{p} - x\right| \leq \frac{1}{p^a}\right\}.
\]
For $M$ a sufficiently large positive integer, and $p_1$ and $p_2$ primes such that
$M< p_1<p_2<2M$, the sets $G_{p_1}(a)$ and $G_{p_2}(a)$ are disjoint.  In fact, the distance between
any point in $G_{p_1}(a)$ and any point in $G_{p_2}(a)$ is greater than or equal to
\[
\frac{1}{4M^2}-\frac{2}{M^a} \geq \frac{1}{8 M^2}.
\]
For such $M,$ the set
\[
K_M(a)=\bigcup_{\substack{p \text{ prime}\\M<p< 2M}} G_p(a)
\]
is the disjoint union of the intervals composing the sets $G_p(a)$.  So $K_{M}(a)$ is made up of
intervals of length less than or equal to~$2/M^a$ which are separated by gaps of length at least~$1/(8M^2)$.

We obtain \Jarnik's fractal by choosing a sequence $\vec{m}=(m_i:i\geq 1)$, which is sufficiently
fast growing in a sense to be determined below. Each number $m_i$ will play the role of an $M$ as described in the previous paragraph. We let $E^J_0(\vec{m},a)=[0,1]$ and for
$k=1,2,\ldots$ let $E^J_k(\vec{m},a)$ be the union of those intervals of $K_{m_k}(a)$ that are
completely contained in $E^J_{k-1}(\vec{m},a)$.  By discarding a negligible number of intervals, we
arrange that all intervals from $E^J_{k-1}(\vec{m},a)$ are split into the same number of intervals
in $E^J_k(\vec{m},a)$.  Let $i_k$ be the number of intervals from $E^J_k(\vec{m},a)$ which are
contained in a single interval of $E^J_{k-1}(\vec{m},a)$.  Let
$E^J(\vec{m},a)=\cap_{k\geq 1}E^J_k(\vec{m},a)$.  We define a mass distribution $\mu_J(\vec{m},a)$
on $E^J(\vec{m},a)$ by assigning a mass of $1/(i_1 \times \ldots \times i_{k})$ to each of the
$i_1 \times \ldots \times i_k$ many $k$-level intervals in $E^J_k(\vec{m},a).$ Then
$\mu_J(\vec{m},a)$ has the mass distribution property for
exponent~$2/a$,~\citep[see][Chapter~10]{Fal2003}.

There are two ways by which $E^J(\vec{m},a)$ can be thinned to a subset of lower Hausdorff
dimension.  First, we could use fewer numbers between $m_k$ and $2m_k$ when we define
$E^J_k(\vec{m},a)$.  However, even if we choose only one denominator at each level, the resulting set
has dimension~$1/a$ \citep[see][Example~4.7]{Fal2003}.  To obtain a dimension smaller than $1/a$,
not only do we choose just one denominator at each level, but we also choose only the intervals centered
on a uniformly spaced subset of the rational numbers with that denominator.

There is a further variation on \Jarnik's construction and the above thinned version of it which
allows for approximating $a$ and~$b$. 
To express it we  introduce the following definition.

\begin{definition}
The sequences of real numbers   $\vec{a}$ and $\vec{b}$
are \emph{appropriate} when $\vec{a}=(a_k:k\in\N)$ is non-decreasing with $a_k \geq 1$ for all $k$ and limit $a$ greater than or
equal to~$2$ and $\vec{b}=(b_k:k\in\N)$ is strictly increasing with limit $b$ less than or equal
to~$2/a$ such that if $1/a<b$ then $1/a<b_1$.
\end{definition}

For appropriate $\vec{a}$ and $\vec{b}$, we can
modify \Jarnik's fractal to accommodate the specification of $a$ in the limit by substituting $a_k$
in place of $a$ in the definition of $E_k^J$.  That is, we let $E^J_k(\vec{m},\vec{a})$ be the
collection of intervals of $K_{m_k}(a_k)$ which are completely contained in
$E^J_{k-1}(\vec{m},\vec{a})$, adjusted by removing intervals so that every interval in
$E^J_{k-1}(\vec{m},\vec{a})$ contains the same number of intervals in $E^J_k(\vec{m},\vec{a})$.  By
construction, the intervals in $E_k^J(\vec{m},\vec{a})$ are of the form
$\left[\frac{q}{p}-\frac{1}{p^{a_k}},\frac{q}{p}+\frac{1}{p^{a_k}}\right].$ 
It follows that every real number in $E^J(\vec{m},\vec{a})$ has
irrationality exponent greater than or equal to~$a$.  Further, when the sequence $\vec{m}$ grows
sufficiently quickly, the uniform measure $\mu_J(\vec{m},\vec{a})$ on $E^J(\vec{m},\vec{a})$ has the
mass distribution property for exponent~$2/a$.  It follows that $\mu_J(\vec{m},\vec{a})$-almost
every real number has irrationality exponent exactly equal to~$a$.  Similarly, we can modify the way
that we thin $E^J(\vec{m},\vec{a})$ to reduce dimension from $2/a$ to $b$.  The construction is not
sensitive on this point and using $b_k$ to determine how to thin at step $k$ results in a fractal of
dimension~$b.$

\subsection{Irrationality exponent~$a$, Hausdorff Dimension $b$ and $0<b\leq 1/a$}

\begin{definition}[Family of fractals $\E(\vec{q},\vec{m},\vec{a})$ ]\label{5}
  Let $\vec{m}$ be an increasing sequence of positive integers; let $\vec{q}$ be a sequence of
  integers;
  let $\vec{h}$ be a sequence of integers such that for each $k,$ $h_k\in[0,q_k)$; and let $\vec{a}$
  be a non-decreasing sequence of real numbers greater than or equal to~$2$ with limit~$a$.
    \begin{itemize}
      \item Let $\E_{1}(\vec{h},\vec{q},\vec{m},\vec{a})$ be $[0,1]$.   
      \item Given $\E_{k-1}(\vec{h},\vec{q},\vec{m},\vec{a}),$ let
        $\E_{k}(\vec{h},\vec{q},\vec{m},\vec{a})$ be the collection of intervals in  $G_{m_k}(a_k)$
        which are completely contained in intervals from
        $\E_{k-1}(\vec{h},\vec{q},\vec{m},\vec{a})$ and which are of the form
        $\left[\frac{r}{m_k}-\frac{1}{m_k^{a_k}},\frac{r}{m_k}+\frac{1}{m_k^{a_k}}\right]$ 
        such that $r\equiv h_k\mod{q_k}$.  
        As usual, discard a negligible number of intervals so that each interval in
        $\E_{k-1}(\vec{h},\vec{q},\vec{m},\vec{a})$ has the same number of subintervals in
        $\E_{k}(\vec{h},\vec{q},\vec{m},\vec{a})$.  Further, ensure that this number of subintervals is
        independent of~$\vec{h}$.  
    \end{itemize}
    Let $\E(\vec{q},\vec{m},\vec{a})$ be the family of fractals obtained by considering all
    possible sequences~$\vec{h}$.
\end{definition}

By construction, if $\vec{h}$ and $\vec{g}$ have the same first $k$ values, then for all $j\leq k$,
$\E_{j}(\vec{h},\vec{q},\vec{m},\vec{a})$ is equal to $\E_{j}(\vec{g},\vec{q},\vec{m},\vec{a})$.
So, $\E(\vec{q},\vec{m},\vec{a})$ is actually a finitely-branching tree of fractals.

\begin{lemma}\label{6}
  Suppose that $\vec{a}$ and $\vec{b}$ are appropriate sequences of reals  with limits $a$ and $b$ such that
  $a\geq 2$ and $b\leq 1/a$.  There is a function $f$, computable from $\vec{a}$ and $\vec{b},$ such that for any
  sequence of integers $\vec{m}=(m_k:k\in\N)$ for which for all $k$, $m_{k+1}\geq f(k,m_k)$, there is a sequence
  of integers $\vec{q}$, such that for all $E\in \E(\vec{q},\vec{m},\vec{a})$ and for $\mu$ the uniform
  measure on $E$, the following conditions hold.
  \begin{itemize}
  \item For all $k$ greater than $2$ and all intervals $I$ such that 
    $|I|\leq 
   \frac{q_{k-1}}{m_{k-1}}-\frac{2}{m_{k-1}^a}$, 
    \[\mu(I)<|I|^{b_k}.
     \]
  \item For all integers $k$, 
   \[
   m_k^{a_k b_{k+1}-1}\leq \frac{1}{q_k}\leq m_k^{a_k b_{k+2}-1}.
   \]
    \end{itemize}
    Further, we can compute $q_k$ from $(a_1,\dots,a_k)$, $(b_1,\dots,b_{k+2})$ and
    $(m_1,\dots,m_k)$.
\end{lemma}

\begin{proof}
  We consider $E\in \E(\vec{q},\vec{m},\vec{a})$ and $\mu$ defined on $E$ as above from $\vec{a}$,
  an increasing sequence $\vec{m}=(m_k:k\geq 1)$ and $\vec{q}=(q_k: k\geq 1).$   For a given interval $I$, we
  estimate $\mu(I)$ and we deduce a sufficient growth rate on $\vec{m}$ and appropriate values for
  $\vec{q}$ in terms of $\vec{m}$ so as to ensure the desired inequality $\mu(I)<|I|^{b_k}$,
  for $I$ as specified.  The existence of $f$ follows by observing that these functions are
  computable from $\vec{a}$ and~$\vec{b}$.

  We follow a modified version of the proof of \Jarnik's Theorem as presented in \cite{Fal2003}.  We
  take $m_1$ to be larger than $3 \times 2^a$ and sufficiently large so that $1/m_1>2m_1^{-a}$. 

We let $E_0=[0,1]$ and for
$k\geq 1$ let $E_k$ consist of those intervals of $G_{m_k}(a_k)$ that are completely
contained in $E_{k-1}$ and are of the form $\left[\frac{r}{m_k}-\frac{1}{m_k^{a_k}},\frac{r}{m_k}+\frac{1}{m_k^{a_k}}\right]$ such that $r\equiv h_k\mod{q_k}$. 
Thus, the intervals of $E_k$ are of length $2/m_k^{a_k}$ and are separated by gaps of length at least
  \[
  g_k=\frac{q_k}{m_k}-\frac{2}{m_k^{a_k}}.
  \] 
  Let $i_k$ be the number of intervals of $E_{k}$ contained in a single interval of $E_{k-1}$.  By
  construction $i_1=m_1/q_1$ and for every~$k>1$,
  \begin{equation}
    \label{eq:1}
  i_k\geq \frac{1}{2} \frac{2}{(m_{k-1})^{a_{k-1}}}\,\frac{1}{q_k}\, m_k = \frac{\
    m_k}{(m_{k-1})^{a_{k-1}} q_k},
  \end{equation}
  which represents half of the product of the length of an interval in $E_{k-1}$ and the number of
  intervals with centers $p/m_k,$ where $p$ is an integer with fixed residue modulo $q_k$.  This
  estimate applies provided that $q_k$ is less than $m_k$ and $m_k$ is sufficiently large with
  respect to the value of~$m_{k-1}$.

  Now, we suppose that $S$ is a subinterval of $[0,1]$ of length $|S|\leq g_1$ and we estimate~$\mu(S)$.  Let
  $k$ be the integer such that $g_k\leq |S| < g_{k-1}$.  The number of $k$-level intervals that
  intersect~$S$~is
  \begin{itemize}
  \item  at most  $i_{k}$, since $S$ intersects at most one $(k-1)$-level interval,
  \item  at most $2+ |S|  {m_k}/{q_k} \leq 4|S|m_k/q_k$, by an
    estimate similar to that for the lower bound on~$i_k$.
  \end{itemize}

  Each $k$-level interval has measure $1/(i_1\times\ldots\times i_k)$.
  We have that 
  \begin{align}
    \mu(S)&\ \leq \ \frac{\min(4|S| m_k/q_k,i_k)}{i_1\times\ldots\times i_k}\label{eq:2} \\
          & \ \leq \ \frac{ (4|S| m_k/q_k)^{ b_k} \  i_{k}^{1-b_k}}{i_1\times\ldots\times
            i_k}\nonumber \qquad \text{(as $b_k \in [0,1]$)}\\
          &\ =\  \frac{4^{b_k} m_k^{b_k}}{(i_1\times \ldots\times i_{k-1}) \ i_k^{b_k} q_k^{b_k} }
            |S|^{b_k}\nonumber\\
          &\ \leq\    1      \frac{ m_{1}^{a_1} q_2  }{m_2}\ \ \frac{m_2^{a_2} q_3  }
      {m_{3}}
      \ldots 
      \frac
      { m_{k-2}^{a_{k-2}} q_{k-1}}
      {m_{k-1}}
      \ \
            \frac{4^{b_k}m_k^{{b_k}}}{i_k^{b_k} q_k^{b_k}} \ |S|^{b_k}\nonumber    \qquad \text{(by~\eqref{eq:1})}\\
    &\  \leq  \  \frac
      { m_{1}^{a_1} q_2  }
      {m_2}\ \
      \frac{m_2^{a_2} q_3  }
      {m_{3}}
      \ldots 
      \frac
      { m_{k-2}^{a_{k-2}} q_{k-1}}
      {m_{k-1}}
      \ \ 
      \left(\frac{ m_{k-1}^{a_{k-1}} q_k  }{m_k}\right)^{b_k}
      \frac{m_k^{{b_k}}}{q_k^{b_k}}
      4^{b_k}\
      |S|^{b_k}\nonumber \\
    & \  = \ (q_2\cdots q_{k-1})
      \big(m_{1}^{a_1} m_{2}^{a_2-1}  \ldots m_{k-2}^{a_{k-2}-1}\big)
      4^{b_k} m_{k-1}^{a_{k-1} b_k-1 }  
      |S|^{b_k}.\label{eq:3}  
  \end{align}

  We want to ensure that for all such $S$, $\mu(S) < |S|^{b_k}$, from which we may infer that
  $\mu(S) < |S|^{s}$ for all $s$ in $[0,b_k]$.  Thus, it suffices to show that there is a suitable
  growth function $f$ for the sequence $\vec{m}$, and there is a suitable sequence of values for
  $\vec{q}$ such that
  \begin{equation*}
(q_2\cdots q_{k-2}    q_{k-1})
    \big(m_{1}^{a_1} m_{2}^{a_2-1}  \ldots m_{k-2}^{a_{k-2}-1}\big)
    4^{b_k} m_{k-1}^{a_{k-1}b_k-1}
    < 1.
  \end{equation*}
  Equivalently,
  \begin{equation*}
    (q_2\cdots q_{k-2})
    \big(m_{1}^{a_1} m_{2}^{a_2-1}  \ldots m_{k-2}^{a_{k-2}-1}\big)
    4^{b_k}m_{k-1}^{a_{k-1}b_k-1} <  \frac{1}{q_{k-1}}.
  \end{equation*}
  If we let $C$ be the term that does not depend on $m_{k-1}$ or on $q_{k-1}$, we can satisfy the
  first claim of the Lemma by satisfying the requirement
  \begin{equation*}
   C m_{k-1}^{a_{k-1}b_k-1} <  \frac{1}{q_{k-1}}.\label{eq:4}
 \end{equation*}
 Since $C$ is greater than $1$, this requirement on $m_{k-1}$ and $q_{k-1}$ also ensures part of the
 second claim of the lemma, that $m_{k-1}^{a_{k-1}b_k-1} < 1/q_{k-1}$.  Since
 $b_k<b \leq 1/a\leq 1/a_{k-1}$, by making $m_{k-1}$ sufficiently large and letting $q_{k-1}$ take
 the largest value such that $C m_{k-1}^{a_{k-1}b_k-1} < 1/q_{k-1}$, we may assume that
 \[
C m_{k-1}^{a_{k-1}b_k-1} > \frac{1}{2q_{k-1}}.
\]  
Equivalently, we may assume that
 \[
2C m_{k-1}^{a_{k-1}b_k-1} > \frac{1}{q_{k-1}}.
\]
 The second clause in the second claim of the Lemma is that $1/q_{k-1} <m_{k-1}^{a_{k-1}b_{k+1}-1}$.
 Since $\vec{b}$ is strictly increasing, by choosing $m_{k-1}$ to be sufficiently large, we may
 ensure that 
\[
2C<m_{k-1}^{a_{k-1}(b_{k+1}-b_k)}
\]
 and so
 $m_{k-1}^{a_{k-1}b_k-1}<1/q_{k-1}<m_{k-1}^{a_{k-1}b_{k+1}-1}, $ as required.  Further, by appropriateness of $\vec{a}$ and $\vec{b}$, $b_k$
 is less than or equal to~$1/a_{k-1}$, the value of~$1/q_{k-1}$ can be made arbitrarily small by
 choosing $m_{k-1}$ to be sufficiently large, so the above estimates apply.
\end{proof}

\begin{lemma}\label{7}
  Suppose that $\vec{a}$ and $\vec{b}$ are appropriate sequences with limits $a$ and $b$ such that
  $a\geq 2$ and $0\leq b\leq 1/a$.  Let $f$ be as in Lemma~\ref{6}, $\vec{m}$ be such that for
  all~$k$, $m_{k+1}\geq f(k,m_k)$, and $\vec{q}$ be defined from these sequences as in
  Lemma~\ref{6}.  For every $E$ in $\E(\vec{q},\vec{m},\vec{a}),$ $E$ has Hausdorff dimension~$b$.
\end{lemma}

\begin{proof}
  Let $E$ be an element of $\E(\vec{q},\vec{m},\vec{a})$ and let $\mu$ be the uniform measure on
  $E$.  By Lemma~\ref{6}, for every $\beta<b$, for all sufficiently small intervals $I$,
  $\mu(I)<|I|^\beta.$  Consequently, by application of the Mass Distribution Principle, the Hausdorff
  dimension of $E$ is greater than or equal to~$b$.  Next, consider a real number $\beta$ strictly bigger 
  than~$b$.  For each $k\geq 1$, there are at most $m_k/q_k$ intervals in $E_k$, each
  with radius $1/m_k^{a_k}$.  Then, by Lemma~\ref{6} and the fact that $\vec{a}$ and $\vec{b}$ are appropriate,
\[
\sum_{j\leq {m_k}/{q_k}} \left(\frac{2}{m_k^{a_k}}\right)^\beta
\leq m_k\cdot \frac{2^\beta}{q_k}\cdot m_k^{-a_k\beta}
\leq 2^\beta \cdot  m_k^{ a_k b_{k+2}}\cdot m_k^{-a_k\beta}
\leq \frac{2^\beta}{m_k^{\beta-b_{k+2}}}
\leq\frac{2^\beta}{m_k^{\beta-b}},
\]
  which goes to 0 as $k$ goes to infinity.  It follows that $E$ has Hausdorff dimension less than or
  equal to~$b$, as required to complete the verification of the~Lemma.
\end{proof}

\subsection{Irrationality exponent~$a$, Hausdorff Dimension $b$ and $1/a\leq b\leq 2/a$}

\begin{definition}[Family of fractals $\F(\vec{q},\vec{m},\vec{a})$]\label{8}
  Let $\vec{m}$ be an increasing sequence of positive integers; let $\vec{q}$ be a sequence of
  integers such that for each $k,$ $q_k$ is between $1$ and the cardinality of the set of prime
  numbers in $[m_k,2m_k)$; let $\vec{H}$ be a sequence of subsets of primes such that, for each $k \geq 1$, $H_k$ contains exactly $q_k$ primes from the interval $[m_k, 2m_k )$; and let $\vec{a}$ be a
  non-decreasing sequence of real numbers greater than or equal to~$2$ with limit~$a$.
    \begin{itemize}
      \item Let $\F_1(\vec{H},\vec{q},\vec{m},\vec{a})$ be $[0,1]$.

      \item Given $\F_{k-1}(\vec{H},\vec{q},\vec{m},\vec{a}),$ let
        $\F_{k}(\vec{H},\vec{q},\vec{m},\vec{a})$ be the collection of intervals in 
\[
\bigcup_{p\in H_k} G_p(a_k) ,
\] 
        which are completely contained in intervals from
        $\F_{k-1}(\vec{H},\vec{q},\vec{m},\vec{a})$.

      \item Let $s$ denote the function that maps $k$ to the ratio given by $q_k$, the number of
        retained denominators, divided by the number possible denominators, which is the number of
        primes in $[m_k,2m_k)$.  Discard a negligible number of intervals so that each interval in
        $\F_{k-1}(\vec{H},\vec{q},\vec{m},\vec{a})$ has the same number of subintervals in
        $\F_{k}(\vec{H},\vec{q},\vec{m},\vec{a})$, and further, so that this number of subintervals
        depends only on $k,\vec{q},\vec{m},\vec{a}$ and not on $\vec{H}$.
    \end{itemize}
    Let $\F(\vec{q},\vec{m},\vec{a})$ be the family of fractals obtained by considering all
    possible sequences~$\vec{H}$.
\end{definition}

\begin{lemma}\label{9}
  Suppose that $\vec{a}$ and $\vec{b}$ are appropriate sequences with limits $a$ and $b$ such that
  $a>2$ and $1/a\leq b\leq 2/a$.  There is a function~$f$, computable from $\vec{a}$
  and~$\vec{b},$ such that for any sequence $\vec{m}=(m_k:k\in\N)$ for which for all~$k$,
  $m_{k+1}\geq f(k,m_k)$, there is a sequence $\vec{q}$, such that for all
  $E\in \F(\vec{q},\vec{m},\vec{a})$ and for $\mu$ the uniform measure on~$E$, the following conditions hold:
 \begin{itemize}
  \item For all $k>2$ and all intervals $I$ such that $|I|\leq    \frac{1}{4m_{k-1}^2}-\frac{2}{m_{k-1}^a},$ 
   \[
     \mu(I) <|I|^{b_k}.
   \]
 \item For all  integers $k,$
    \[
    \log(m_{k})m_k^{a_k b_{k+1}-2}\leq \frac{1}{q_k}\leq \log(m_{k})m_k^{a_k b_{k+2}-2}.
    \]
\end{itemize}
  Further, we can exhibit such an $\vec{q}$ for which $q_k$ is uniformly computable from
  $(a_1,\dots,a_k)$, $(b_1,\dots,b_k)$ and $(m_1,\dots,m_{k-1})$.
\end{lemma}

\begin{proof}
  The proof of Lemma~\ref{9} has the same structure as the proof of Lemma~\ref{6}, with fundamental
  difference as follows.  Lemma~\ref{6} refers to $\E$, the tree of subfractals of
  $E^J(\vec{m},\vec{a})$, where the subfractals are obtained by recursion during which at step $k$
  only $1/q_k$ of the $m_k$-many eligible intervals in $G_{m_k}(a_k)$ are used. Lemma~\ref{9}
  refers to $\F$, the tree of subfractals  of $E^J(\vec{m},\vec{a})$,
  where the subfractals are obtained by recursion during which at step $k$ for only $1/q_k$ of the
  eligible denominators $m$, all of the intervals in $G_{m}(a_k)$ are used.  The set of eligible
  denominators is the set of prime numbers between $m_k$ and $2m_k$.  By choosing $m_0$ large
  enough, the prime number theorem implies that this set of eligible denominators has  between
  $m_k/2\log(m_k)$ and $2m_k/\log(m_k)$ many elements and each contributes at least $m_k$ many
  eligible intervals.  Now, we give an abbreviated account to indicate how this difference
  propagates through the proof.

 We consider $E\in\F(\vec{q},\vec{m},\vec{a})$ and $\mu$ the uniform measure on $E$.  We let $i_k$
 be the number of intervals in $E_k$ contained in a single interval of $E_{k-1}$.  Now, we have the
 following version of Inequality~(\ref{eq:1}),
 \[
 i_k\geq
 \frac12\,\frac{2}{(2m_{k-1})^{a_{k-1}}}\,\frac{1}{q_k}\,\frac{m_k^2}{2\log(m_k)}>
  \frac{m_k^2}{2^{a+1}\,(m_{k-1})^{a_{k-1}}\,\,\log(m_k)\,q_k}.
 \]
The intervals in $E_k$ are separated by gaps of length at least $g_k=1/4m_k^2-(2/m_k^{a_k})$,
because given two
intervals with centers $c_1/d_1$ and $c_2/d_2$, the gap between them is
\[
\left|\frac{c_1d_2-c_2d_1}{d_1d_2}\right|-\left(\frac{1}{c^{a_k}}+\frac{1}{d^{a_k}}\right).
\]  
The numerator in the first
fraction is at least $1$ and the denominator is no greater than~$(2m_k)^2$.  The denominators in
the second term are at least~$m_k^{a_k}$.  

Now we suppose that $S$ is a subinterval of $[0,1]$ of length $|S|\leq g_1$ and we estimate~$\mu(S)$.
Let $k$ be the integer such that $g_k\leq |S| < g_{k-1}$.  The number of $k$-level intervals that
intersect~$S$~is
  \begin{itemize}
  \item  at most  $i_{k}$, since $S$ intersects at most one $(k-1)$-level interval,
  \item at most $2+ |S| {4m_k^2}/( q_k\log(m_k))$  which is less than $8|S|m^2_k/(q_k\log(m_k))$.
  This is because there are  at most $2m_k/\log{m_k}$ primes between $m_k$ and
  $2m_k$ and each prime contributes at most $2m_k$ many intervals in $[0,1]$.  
  And we keep $1/q_k$ of these.
\end{itemize}
  As before, each $k$-level interval has measure $1/(i_1\times\ldots\times i_k)$.  Then, we have a version
  of Inequality~(\ref{2}).
  \[
  \mu(S)\leq \frac{\min\left(\frac{8|S|m^2_k}{q_k\log(m_k)}, i_k\right)}{i_1\times\ldots\times i_k}.
  \]
  Upon manipulation as before, we have a version of Inequality~(\ref{eq:3}).
  \[
  \mu(S)\leq 8^{b_k}2^{k(a+1)} (q_2\cdots q_{k-1})\,(\log(m_2)\cdots\log(m_{k-1})\,(m_1^{a_1}m_2^{a_2-2}\cdots m_{k-2}^{a_{k-2}-2})m_{k-1}^{a_{k-1}b_{k}-2}|S|^{b_k}.
  \]
 To ensure that $\mu(S)<|S|^{b_k}$, we must find a suitable growth function $f$ for the sequence
 $\vec{m}$ so that there is a suitable sequence of values for $\vec{q}$ such that
 \[
 8^{b_k}2^{k(a+1)} (q_2\cdots q_{k-1})\,(\log(m_2)\cdots\log(m_{k-1}))\,(m_1^{a_1}m_2^{a_2-2}\cdots
 m_{k-2}^{a_{k-2}-2})m_{k-1}^{a_{k-1}b_{k}-2}<1.
 \]
 Equivalently,
 \[
 8^{b_k}2^{k(a+1)} (q_2\cdots q_{k-2})\,(\log(m_2)\cdots\log(m_{k-1}))\,(m_1^{a_1}m_2^{a_2-2}\cdots
 m_{k-2}^{a_{k-2}-2})m_{k-1}^{a_{k-1}b_{k}-2}<1/q_{k-1}.
 \]
 Let $C$ be the term that does not depend on $m_{k-1}$ or $q_{k-1}$.  We can satisfy the first claim
 of Lemma~\ref{9} by satisfying the requirement
 \[
 C\log(m_{k-1}) m_{k-1}^{a_{k-1}b_k-2}<1/q_{k-1}.
 \]
 Since $C$ is greater than $1$, this requirement on $m_{k-1}$ and on $q_{k-1}$ also ensures part of
 the second claim of the lemma, that $\log(m_{k-1})m_{k-1}^{a_{k-1}b_k-2}<1/q_{k-1}$.  Since $b_{k}\leq
 2/a<2/a_{k-1}$, $a_{k-1}b_k-2$ is negative.  By making $m_{k-1}$ sufficiently large and letting 
 $q_{k-1}$ take the largest value  such that $ C\log(m_{k-1}) m_{k-1}^{a_{k-1}b_k-2}<1/q_{k-1}$, we
 may assume that
 \[
 C\log(m_{k-1}) m_{k-1}^{a_{k-1}b_k-2}>1/2q_{k-1}.
 \]
 The second clause in the second claim of the lemma is that $1/q_{k-1} <m_{k-1}^{a_{k-1}b_{k+1}-2}$.
 Since $\vec{b}$ is strictly increasing, by choosing $m_{k-1}$ to be sufficiently large, we may
 ensure that 
\[
2C<\log(m_{k-1})m_{k-1}^{a_{k-1}(b_{k+1}-b_k)}
\]
 and so
 $\log(m_{k-1})m_{k-1}^{a_{k-1}b_k-1}<1/q_{k-1}<\log(m_{k-1})m_{k-1}^{a_{k-1}b_{k+1}-2}, $ as
 required.  Further, since $b_k$ 
 is less than or equal to~$1/a_{k-1}$, the value of~$1/q_{k-1}$ can be made arbitrarily small by
 choosing $m_{k-1}$ to be sufficiently large, so the above estimates apply. 
\end{proof}

\begin{lemma}\label{10}
  Suppose that $\vec{a}$ and $\vec{b}$ are appropriate sequences of real numbers with limits $a$ and $b$ such that
  $a>2$, $ 1/a \leq b\leq 2/a$ and $1/a_1\leq b_1$.  Let $f$ be as in Lemma~\ref{9}, $\vec{m}$ be
  such that for all~$k$, $m_{k+1}\geq f(k,m_k)$, and $\vec{q}$ be defined from these sequences as in
  Lemma~\ref{9}.  For every $E$ in $\F(\vec{q},\vec{m},\vec{a}),$ $E$ has Hausdorff dimension~$b$.
\end{lemma}

\begin{proof}
   Let $E$ be an element of $\F(\vec{q},\vec{m},\vec{a})$ and let $\mu$ be the uniform measure on
  $E$.  By Lemma~\ref{9}, for every $\beta<b$, for all sufficiently small intervals $I$,
  $\mu(I)<|I|^\beta.$  Consequently, by application of the Mass Distribution Principle, the Hausdorff
  dimension of $E$ is greater than or equal to~$b$.  Next, consider a real number $\beta$ greater
  than~$b$.  For each $k\geq 1$, there are at most $4m_k^2/(\log(m_k)q_k)$ intervals in $E_k$, each
  with radius less than or equal to $1/m_k^{a_k}$.  Then, by Lemma~\ref{9} and the fact that $a_k \geq 1$,
  \begin{eqnarray*}
\sum_{j\leq 4m_k^2/(\log(m_k)q_k)} \left(\frac{2}{m_k^{a_k}}\right)^\beta&\leq&
                                                                           \frac{4m_k^2}{\log(m_k)}\cdot
                                                                           \frac{1}{q_k}\cdot
                                                                           2^\beta m_k^{-a_k\beta}\\ 
&\leq& \frac{4m_k^2}{\log(m_k)}\cdot \log(m_{k})m_k^{a_k b_{k+2}-2}\cdot 2^\beta m_k^{-a_k\beta}\\
&\leq& \frac{4 \cdot 2^\beta}{m_k^{a_k(\beta-b_{k+2})}}\\
&\leq&\frac{4 \cdot 2^\beta}{m_k^{\beta-b}},
  \end{eqnarray*}
  which goes to 0 as $k$ goes to infinity.  It follows that $E$ has Hausdorff dimension less than or
  equal to~$b$, as required to complete the verification of the~Lemma.
\end{proof}

\subsection{Irrationality Exponent 2, Hausdorff Dimension $b$, and $0< b\leq 1$}

The next case to consider is that in which the desired irrationality exponent is equal to  $2$ and
the desired effective Hausdorff dimension is equal to $b\in (0,1]$.  In fact, we need only consider
the case of $b<1$, since almost every real number with respect to Lebesgue measure has both irrationality
exponent equal to~$2$ and effective Hausdorff dimension equal to~$1$.

\begin{definition}[Family of fractals $\G(\vec{q},\vec{m})$]\label{11}
  Let $\vec{m}$ be an increasing sequence of positive integers such that for all $k$, $m_k$ divides
  $m_{k-1}$; let $\vec{q}$ be a sequence of
  integers;
  and let $\vec{h}$ be a sequence of integers such that for each $k,$ $h_k\in[0,q_k)$.
    \begin{itemize}
      \item Let $\G_{1}(\vec{h},\vec{q},\vec{m})$ be $[0,1]$.   
      \item Given $\G_{k-1}(\vec{h},\vec{q},\vec{m}),$ let
        $\G_{k}(\vec{h},\vec{q},\vec{m})$ be the collection of intervals 
        which are completely contained in intervals from
        $\G_{k-1}(\vec{h},\vec{q},\vec{m})$ and which are of the form
        $\left[\frac{r}{m_k},\frac{r}{m_k}+\frac{1}{m_k}\right]$ such that $r\equiv h_k\mod{q_k}$.  
        As usual, discard a negligible number of intervals so that each interval in
        $\G_{k-1}(\vec{h},\vec{q},\vec{m})$ has the same number of subintervals in
        $\G_{k}(\vec{h},\vec{q},\vec{m})$.  Further, ensure that this number of subintervals is
        independent of~$\vec{h}$.  
    \end{itemize}
    Let $\G(\vec{q},\vec{m})$ be the family of fractals obtained by considering all
    possible sequences~$\vec{h}$.
\end{definition}

In the following Lemma~\ref{12} is parallel to Lemmas \ref{6} and \ref{9}.  Likewise, the next
Lemma~\ref{13} is parallel to Lemmas \ref{7} and \ref{10}.  In fact, the estimates here are simpler
than the earlier ones. We leave the proofs for the interested reader.

\begin{lemma}\label{12}
  Suppose $\vec{b}$ is strictly increasing sequence of real numbers with limit $b$ such that
  $0< b<1$.  There is a function $f$, computable from $\vec{b},$ such that for any sequence
  $\vec{m}=(m_k:k\in\N)$ for which for all $k$, $m_{k+1}\geq f(k,m_k)$, there is a sequence
  $\vec{q}$, such that for all $E\in \G(\vec{q},\vec{m})$ and for $\mu$ the uniform measure on $E$,
  the following conditions hold.
  \begin{itemize}
  \item For all $k$ greater than $2$ and all intervals $I$ such that 
    $|I|\leq 
    \frac{q_{k-1}-1}{m_{k-1}}$, 
    \[\mu(I)<|I|^{b_k}.
     \]
  \item For all integers $k$, 
   \[2
   m_k^{2b_{k+1}-1}\leq \frac{1}{q_k}\leq m_k^{2b_{k+2}-1}.
   \]
    \end{itemize}
    Further, we can compute $q_k$ from $(b_1,\dots,b_{k+2})$ and $(m_1,\dots,m_k)$.
\end{lemma}

\begin{lemma}\label{13}
  Suppose $\vec{b}$ is strictly increasing sequence of real numbers with limit $b$ such that
  $0< b<1$.  Let $f$ be as in Lemma~\ref{12}, $\vec{m}$ be such that for all~$k$,
  $m_{k+1}\geq f(k,m_k)$, and $\vec{q}$ be defined from these sequences as in Lemma~\ref{12}.  For
  every $E$ in $\G(\vec{q},\vec{m}),$ $G$ has Hausdorff dimension~$b$.
\end{lemma}

\section{Hausdorff Dimension and Irrationality Exponent}

In this section, we will put together the pieces and prove Theorem~\ref{1}.

%
%
%

In the proof we will use the following definition.

\begin{definition} For positive integers $d, d_1, d_2, a^*$, we let
\begin{eqnarray*}
    B(d_1,d_2,a^*)&=& \bigcup_{\substack{p,q \in \N \\ d_1\leq q\leq d_2 }} 
\left[\frac{p}{q}-\frac1{q^{a^*}},\frac{p}{q}+\frac1{q^{a^*}}\right]
\\
    B(d,\infty,a^*)&=& \bigcup_{\substack{ p,q\in \N \\  d\leq q }}
\left[\frac{p}{q}-\frac1{q^{a^*}},\frac{p}{q}+\frac1{q^{a^*}}\right] 
  \end{eqnarray*}
 \end{definition}

 \begin{proof}[Proof of Theorem \ref{1}] 
  If $b=0$ the desired set $E$ is quite trivial:   for every $a$ greater than or equal to $2$, including
  $a=\infty$, there is a real number $x$ such that $a$ is the irrationality exponent of $x$.  Let
  $E=\{x\}$ and note that $E$ has Hausdorff dimension equal to~$0$ and the uniform measure on~$E$
  concentrates on a set of real numbers of irrationality exponent~$a$. 

   Assume that $b>0$.  Let $\vec{a}$ be the constant sequence with values $a$ and let $\vec{b}$ be a
   strictly increasing sequence of positive rational numbers with limit $b$.  Thus, $\vec{a}$ and
   $\vec{b}$ are appropriate.  The desired set $E$ will be an element of $\E(\vec{q},\vec{m},\vec{a}),$
   $\F(\vec{q},\vec{m},\vec{a})$ or $\G(\vec{q},\vec{m})$, for $\vec{m}$ and $\vec{q}$ constructed
   according to Lemma~\ref{7},~\ref{10} or~\ref{12}, depending on whether $a>2$ and $b\in(0,1/a)$, 
   or $a>2$ and $b\in[1/a,2/a)$, or $a=2$ and $b\in(0,1)$, respectively.  Since the first claim of the Theorem follows
   from Lemmas~\ref{7},~\ref{10} or~\ref{13}, we need only check the second claim.  We give a full
   account of the case $a>2$ and $0< b\leq1/a$.  We leave it to the reader to note that the same
   argument applies in the other cases.

  Suppose that $b\in(0,1/a)$ and let $f$ and $\vec{q}$ be the functions obtained in Lemma~\ref{6}.  Let
  $\vec{m}$ be the sequence defined by letting $m_1$ be sufficiently large in the sense of
  Lemma~\ref{6} and letting $m_k$ be the least $m$ such that $m$ is greater than $f(k,m_{k-1})$.
  Let $E^J(\vec{m},\vec{a})$ be the \Jarnik-fractal determined from $\vec{m}$ and $\vec{a}$.  Let
  $E^J_k(\vec{m},\vec{a})$ denote the set of intervals used in the definition of
  $E^J(\vec{m},\vec{a})$ at step $k$.  Let $g_k$ denote the minimum gap between two intervals in
  $E^J_k(\vec{m},\vec{a})$.
  We define the sequence $k_s$, $E_{k_s}$ and $d_s$ by recursion on $s$.  Let $k_1=1,$ let
  $E_1=E^J_1(\vec{m},\vec{a})=[0,1],$ and let $d_1$ be the least integer $d$ such that
  $\mu_J(B(d,\infty,a+1/2))$ is less than $\mu_J(E_1)/2^1=1/2$. 
  Now, suppose that $k_s$, $E_{k_s},$ and $d_s$ are defined so that $E_{k_s}$ is an initial segment of
  the levels of an element of $\E(\vec{q},\vec{m},\vec{a})$ and so that
  \[
  \mu_J\big(B\big(d_s,\infty,a+\frac{1}{2^s}\big)\cap E_{k_s}\big)< \frac{\mu_J(E_{k_s})}{2^s}.
  \]
  Let $d_{s+1}$ be the least $d$ such that
  \[
  \mu_J\big(B\big(d,\infty,a+\frac{1}{2^{s+1}}\big)\big)< \mu_J(E_{k_s})/(4\cdot 2^{s+1}).
  \]
  Let $c$ be the number of  intervals in $B(d_s,d_{s+1},a+1/2^s)$. 
  Let $k_{s+1}$ be the least $k$ such that the union of $2c$
  many of the intervals in $E^J_k(\vec{m},\vec{a})$ has measure less than
  $\mu_J(E_{k_s})/(4\cdot 2^s),$ and 
  let $C$ be the collection of intervals in $E^J_{k_{s+1}}$ which
  contain at least one endpoint of an interval in $B(d_s,d_{s+1},a+1/2^s)$. 
  
  Consider the set
  $F_{s+1}$ of extensions of the branch in $\E(\vec{q},\vec{m},\vec{a})$ with endpoint $E_{k_s}$
  to branches of length $k_{s+1}$.  Each element $F_{s+1}$ specifies a set of intervals
  $S$ at level $k_{s+1}.$ 
  Distinct elements in   $F_{s+1}$ have empty intersection and identical $\mu_J$-measure.  
  At most one-fourth of the elements $S$  in $F_{s+1}$ can be such that
  \[
  \mu_J\big(B\big(d_s,\infty,a+\frac{1}{2^s}\big)\cap S \big)\geq 4\cdot \frac{\mu_J(S)}{2^s}.
  \]
  Similarly, at most one-fourth of the $S$ in $F_{s+1}$ can be such that 
  \[
  \mu_J\big(B\big(d_{s+1},\infty,a+\frac{1}{2^{s+1}}\big)\cap S\big)\geq 
   4\cdot \frac{\mu_J(S)}{4\cdot 2^{s+1}} = 
  \frac{\mu_J(S)}{2^{s+1}},
  \]
  and at most one-fourth of the $S$ in $F_{s+1}$ can have 
  \[
    \mu_J(C\cap S)\geq 4\cdot \frac{\mu_J(S)}{4\cdot 2^s}= \frac{\mu_J(S)}{2^s}.
  \]
  Choose one element  $S$ of $F_{s+1}$ which does not belong to any of these fourths and define $E$ through its
  first $k_{s+1}$ levels so as to agree with that element.  Note that
  \[
  \mu_J\big(B\big(d_{s+1},\infty,a+\frac{1}{2^{s+1}}\big)\cap E_{k_{s+1}}\big)\leq  \frac{\mu_J(E_{k_{s+1}})}{2^{s+1}},
  \]
  which was the induction assumption on $E_{k_s}$.  Further note that at most $4/2^s$ of the
  intervals in $E_{k_s}$ can be contained in $B(d_s,d_{s+1},a+1/2^s)$ and at most $1/2^s$ of
  the intervals in~$E_{k_s}$ can contain an endpoint of an interval in $B(d_s,d_{s+1},a+1/2^s)$.
  Thus, the uniform measure $\mu$ on any element of $\E(\vec{q},\vec{m},\vec{a})$ extending the branch up to~$E_{k_{s+1}}$
  assigns $B(d_s,d_{s+1},a+1/2^s)$ a measure of less than $5/2^s$.

  Let $E$ be the set defined as above and for each $k$, let $E_k$ denote the level-$k$ of~$E$, and let
  $\mu$ denote the uniform measure on~$E$.  The first claim of the theorem, that the Hausdorff
  dimension of~$E$ is equal to~$b,$ follows from Lemma~\ref{7}.  For the second claim, every element
  of~$E^J$, and hence of~$E$, has irrationality exponent greater than or equal to~$a$.  So, it is
  sufficient to show that for every positive $\epsilon$,  there is an $s$ such
  that $\mu(B(d_s,\infty, a+\epsilon))<\epsilon$.  Let $s$ be sufficiently large so that
  $5/2^{s-1} < \epsilon$.   Then, since 
  \begin{align*}
    B(d_s,\infty, a+\epsilon)&=\bigcup_{t\geq s}B(d_t,d_{t+1},a+\epsilon)\\
    &\subseteq\bigcup_{t\geq s}B\big(d_t,d_{t+1},a+\frac{1}{2^s}\big)\\
    &\subseteq \bigcup_{t\geq s}B\big(d_t,d_{t+1},a+\frac{1}{2^t}\big),
  \end{align*}
  we have
  \begin{align*}
    \mu(B(d_s,\infty, a+\epsilon))&\leq \sum_{t\geq s}\mu\big(B\big(d_t,d_{t+1},a+\frac{1}{2^t}\big)\big)\\
    &\leq \sum_{t\geq s}\frac{5}{2^t}\\
    &\leq \frac{5}{2^{s-1}}.
  \end{align*}
  The desired result follows.
\end{proof}

\section{Effective Hausdorff Dimension and Irrationality Exponent}

This section is devoted to the proof of Theorem~\ref{2}.

\begin{proof}[Proof of Theorem~\ref{2}] As in our discussion of Theorem~\ref{1}, we will consider
  the case $a>2$ and $b\in[0,1/a]$ in detail.  We leave it to the reader to note that with
  straightforward modifications the argument applies to the other two cases.  For $b=0$, the
  argument reduces to finding a singleton set $E$, so the tree of subfractals $\E$ is just the tree
  of elements of $E^J$.  For $a>2$ and $b\in[1/a,2/a]$, $\F$ replaces $\E$.  For $a=2$ and
  $b\in(0,1]$, $\G$ replaces $\E$.

  Our proof follows the outline of the proof of Theorem~\ref{1}.  That is, we will produce a version
  of $E^J$ and $E$ so that the uniform measure $\mu$ on $E$ concentrates on real numbers with
  irrationality exponent~$a$ and so that every element of $E$ has effective Hausdorff dimension $b$.
  However, since we are not assuming that $a$ and $b$ are computable real numbers, we must work with
  rational approximations when ensuring the condition on effective Hausdorff dimension.  This change
  to add effectiveness to the representation of dimension leads us to weaken our conclusions
  elsewhere: we must settle for showing that $x$ has irrationality exponent greater than or equal
  to~$a$ by showing that for every $a^*<a$ there are infinitely many $p$ and $q$ such that
  $|p/q-x|<1/q^{a^*}$.  A similar modification of \Jarnik's construction appears
  in~\cite*{BugBecSla2014}, for different purpose.

  We will construct $\vec{a}$ and $\vec{b}$ so that $\vec{a}$ is non-decreasing with limit $a$ and
  so that $\vec{b}$ is strictly increasing with limit $b$.  Simultaneously, we will construct
  $\vec{m}$, $\vec{q}$ and $E$ in $\E(\vec{q},\vec{m},\vec{a})$ as in the proof of
  Theorem~\ref{1}.  In Theorem~\ref{1}, we began with $E^J(\vec{m},\vec{a})$ and
  $\E(\vec{q},\vec{m},\vec{a})$.  We defined a sequence of integers $d_k$ and an element of
  $\E(\vec{q},\vec{m},\vec{a})$.    At step $k$, we ensured that the set of real numbers with
  irrationality exponent greater than $a+1/2^k$ had small measure with respect to $\mu$.  It would
  have been sufficient to ensure the same fact for the set of numbers with irrationality exponent
  less than $a+\epsilon_k$, provided that $\epsilon_k$ was a non-increasing sequence with limit
  zero, which is how we will proceed now.  Thus, we will construct $(\alphabar_s:s\in\N)$ to be a 
  non-increasing sequence with limit $a$ to stand in for $(a+1/2^s:s\in\N)$. 

  Consider the problem of ensuring that for a non-negative integer $k$ and for every element $x$ of $E$,
  the sequence $\sigma$ consisting of the first $\log(m_{k})a_{k}$ digits in the base-2 expansion of
  $x$ has Kolmogorov complexity less than $\log(m_{k})a_{k} b$.  For this, it would be sufficient
  to exhibit an uniformly computable map taking binary sequences of length less than
  $\log(m_{k})a_{k} b$ onto the set of intervals in $E_k$.  For large enough $m_{k}$, we can use a
  binary sequence of length $\log(m_{k})(b-b_k)a_k$ to describe the first $k-1$ steps of the definition
  of $E^J$ and any initial conditions imposed at the beginning of step $k$.  It will then be
  sufficient to show that this information is enough to compute a map taking binary sequences of
  length less than $\log(m_{k})a_{k} b_k$ onto the set of intervals in $E_k$.

  We proceed by recursion on $s$. For each $s$, we specify three integers, $\ell_s,$ $k_s,$ and
  $d_s$.  We specify $d_s$ as we did in Theorem~\ref{1}.  We will specify rational numbers
  $\alpha_s$, $\overline\alpha_s$ and $\beta_s$ and use them to specify the values of $\vec{a}$, of
  up to $\ell_s$, specify the levels of $E^J$ up to $\ell_s$, which means that we also
  specify sequence of numbers $\vec{m}$ up to $\ell_s$, and we specify the values of $\vec{b}$ up to
  $\ell_s+2$.    This determines the values of $\vec{q}$ up to $\ell_s$.  Finally, we specify the first $k_s$ many
  levels of $E$.

  \subparagraph{Initialization of the recursion.}  Let $\beta_0$ be a positive rational number less
  than $b$.  Let $\epsilon_0$ be a rational number such that $\beta_0+\epsilon_0<b$.  Let
  $\alpha_0$ and $\overline{\alpha}_0$ be positive rational numbers such that
  $\alpha_0<a<\overline{\alpha}_0$ and $\overline\alpha_0-\alpha_0<1$.  
   Let $E_0=E^J_0=[0,1]$; let $m_1$ be sufficiently large in the
  sense of Lemma~\ref{6} for the constant sequence $\vec{\alpha}$ with value $\alpha_0$ and the
  sequence $\vec{\beta}$ with values $\beta_0+(1-1/2^n)\epsilon_0$, respectively; let
  $E^J_1=G_{m_1}(\alpha_0)$, the collection of intervals centered at rational numbers
  $p/m_1\in(0,1)$ with diameter $2/m_1^{\alpha_0}$; 
   let    $  g_1=\frac{q_1}{m_1}-\frac{2}{m_1^{\alpha}}$, which is the minimum distance between two intervals in $E_1^J$;
   and 
   let $d_0$ be the least integer $d$ such that
  $2/d^{\alpha_0}$ is less than $g_1$
  and such that
  \[
  \sum_{j\geq d}\frac{ j}{ (2 j)^{\overline{\alpha}_0}}< \frac{1}{2}.
  \]
  Since $\overline{\alpha}_0 > 2$, $d_0$ is well-defined.  
  Since we will define $E^J$ so that
  Lemma~\ref{6} applies, this sum is an upper bound on $\mu_J(B(d_0,\infty,\alpha_0))$.  So, this
  choice of $d_0$ ensures that $\mu_J(B(d_0,\infty,\alpha_0))$ is less than $1/2$.

  Let $\ell_0=1$; let the first two values of $\vec{a}$ be equal to those in $\vec{\alpha}$ and the
  first $3$ values of $\vec{b}$ be the same as those in $\vec{\beta}$.  
  Note that these choices
  determine $q_1$.  Let $k_0=0$ and let $E_1=E^J_1$.

  \subparagraph{Recursion: stage $s+1$.}  Now, suppose that our construction is defined through
  stage~$s$.  Following the proof of Theorem~\ref{1}, we may assume that we have ensured 
\[
  \mu_J(B(d_{s},\infty,\alpha_s)\cap E_{k_s})) \ < \ \frac{\mu_J(E_{k_s})}{2^{s}},
\] subject to our satisfying the hypotheses of~Lemma~\ref{6}.

  We ensure that stage $s+1$ of the construction of $E$ is uniformly computable from the
  construction up to step $s$ and parameters set during the initialization of stage $s+1$ by
  continuing the construction of $E^J$ recursively in these parameters until the definition of
  $E_{k_{s+1}}$ is evident.
  
  \subsubparagraph{Initialization of stage $s+1$.} Let $\alpha_{s+1},$ $\alphabar_{s+1}$ be rational
  numbers such that
  \[
  \alpha_{s}<\alpha_{s+1}<a<\alphabar_{s+1}\leq \alphabar_{s}
  \]
  and such that 
 \[
  |\alphabar_{s+1}-\alpha_{s+1}|<\frac{1}{s+1}.
  \]  
  Let $\beta_{s+1}$ be a rational number
  strictly between $\beta_{s}+\epsilon_s$ and $b$ such that 
  \[
  |b-\beta_{s+1}|<\frac{1}{s+1},
  \]
 and let
  $\epsilon_{s+1}$ be a rational number such that 
  \[
   \beta_{s+1}+\epsilon_{s+1} <b.
  \]
  Let $d_{s+1}$ be the least integer $d$ such that $2/d^{\alphabar_{s+1}}<g_{\ell_s}$ 
   such that  $2/d^{\alphabar_{s+1}}$  is less than the minimum gap length  
   $g_{\ell_s}= \frac{q_{\ell_s}}{m_{\ell_s}}-\frac{2}{m_{\ell_s}^{\alphabar_s}}$
  and
  \[
  \sum_{j\geq d} \frac{j}{(2 j)^{\alphabar_{s+1}}}<\frac{\mu_J(E_{k_{s}})}{4\cdot 2^{s+1}}.
  \]
  Since $2/d_{s+1}^{\alphabar_{s+1}}<g_{\ell_s}$ and we ensure that our construction of
  $E^J$ satisfies the hypotheses of Lemma~\ref{6}, this sum is an upper bound on
  $\mu_J(B_{d_{s+1}},\infty,\alphabar_{\ell_{s+1}})$.  Thus, we have ensured that
  \[
    \mu_J(B_{d_{s+1}},\infty,\alphabar_{\ell_{s+1}}) < \frac{\mu_J(E_{k_{s}})}{4\cdot 2^{s+1}},
  \]
  which is analogous to how we chose $d_{s+1}$ in the proof of Theorem~\ref{1}.

  Let $m_{\ell_s+1}$ be greater than $f(\ell_{s},m_{\ell_s})$  and  sufficiently large so  that a binary
  sequence of length $\log(m_{\ell_{s+1}})a(b-\beta_{s+1})$ can  describe these
  parameters together with the first $s$ steps of the construction.

  \subsubparagraph{Subrecursion: substage $\ell$.}  Proceed by recursion on substages $\ell$
  starting with initial value $\ell_{s}+1$.  Suppose that the termination condition for stage $s+1$
  was not realized during substage $\ell-1$.  Define $a_{\ell}$ to be equal to $\alpha_{s+1}$ and
  define $b_{\ell+2}$ to be equal $\beta_{s+1}+\epsilon_{s+1}(1-1/2^{\ell-\ell_{s}})$. 
  If~$\ell=\ell_s+1$, then value of $m_{\ell_s+1}$ was assigned in the previous
  paragraph.  Otherwise, let $m_\ell$ be larger than $f(\ell,m_{\ell-1})$.  Note that $f$ is defined in
  terms the values of $\vec{q}$ and $\vec{a}$ up to $\ell$ and that the values of $\vec{q}$ are
  defined in terms of $\vec{a}$ and $\vec{m}$ up to $\ell$ and $\vec{b}$ up to $\ell+2$, all of
  which have been determined before the evaluation of~$f$.

\subsubparagraph{Termination of the subrecursion.}  Note that $\mu_J(E_{k_s})$
  is determined at the end of step $s$, since it is equal to the number of level-$k_s$ intervals in
  $E_{k_s}$ 
  divided by the number of level-$(k_s)$ intervals in $E^J$.
  We say that step $\ell$ satisfies the termination condition for stage $s+1$ when
  there is a $k$ between $\ell_{s}$ and $\ell$ such that there is an $S$ in  the level $k$ of $\E(\vec{q}, \vec{m}, a)$
  satisfying the following conditions. 
  \begin{itemize}
  \item There is a $d^*$ in $(d_{s},d_s+\ell)$ such that the $\mu_J$-measure of the set of intervals in
    $E_\ell^J$ which have non-empty intersection with $B(d_{s},d^*,\alphabar_{s})\cap S$ is
    less than \[
    4\cdot\frac{\mu_J(S)}{2^{s}}- \sum_{j\geq d^*} \frac{j}{ (2 j)^{\alphabar_{s+1}}}.
   \]

  \item There is a $d^*$ in $(d_{s+1},d_{s+1}+\ell)$ such that the $\mu_J$-measure of the set of intervals in
    $E_\ell^J$ which have non-empty intersection with $B(d_{s+1},d^*,\alphabar_{s+1})$ is
    less than 
   \[
    4\cdot\frac{\mu_J(S)}{2^{s+1}}- \sum_{j\geq d^*} \frac{j}{ (2 j)^{\alphabar_{s+1}}}.
    \]

  \item The $\mu_J$ measure of the union of the set of intervals in $S$ 
   which contain at least one endpoint of an interval in $B(d_{s},d_{s+1},\alphabar_{s+1})$ is less than
    \[
    \frac{\mu_J(S)}{  2^{s+1}}.
    \]
  \end{itemize}

  If the termination condition is realized for $\ell$, then we set $\ell_{s+1}$ equal to~$\ell$, and
  set $k_{s+1}$ and $E_{k_{s+1}}$ to have values equal to those in the least pair $k$ and $S$
  which satisfy the termination condition.  

  \paragraph{Verification.}  We first note that by inspection of our construction, it satisfies the hypotheses
  of Lemma~\ref{6}.  
  Next, we observe that for every stage the subrecursion for that stage eventually realizes its
  termination condition.  For the sake of a contradiction, suppose that there is a stage in the main
  recursion of the construction whose subrecursion never terminates.  Then $E^J$ is defined using
  an eventually constant sequence $\vec{a}$ and an increasing sequence $\vec{b}$.  By the same
  reasoning as in the proof of Theorem~\ref{1}: first, there must be a $k$ and an $S$ such that
  the three sets in question have sufficiently small relative measure in $S$; and second, for
  any such $k$ and $S$, for any $\delta>0$, there is an $\ell$ such that the measures of the
  intersections of those three sets with $S$ is approximated to within $\delta$ by the measure
  of their smallest covers using intervals in $E_\ell^J$.  But, then the termination condition would
  apply, a contradiction.  Thus, both $E$ and $E^J$ are well-defined, as are their uniform measures
  $\mu$ and~$\mu_J$.

  By the same argument as in the proof of Theorem~\ref{1}, $\mu$-almost every real number~$x$ has
  exponent of irrationality equal to~$a$.  By Lemma~\ref{3}, the Mass Distribution Principle applied
  to~$\mu$, if $B$ is a subset of the real numbers and the Hausdorff dimension of~$B$ is less than
  $b$, then $\mu(B)=0$.  So, $\mu$-almost every~$x$ has effective Hausdorff dimension greater than
  or equal to~$b$.

  In order to conclude that every element of $E$ has effective Hausdorff dimension less than or
  equal to $b$, it remains to show that for every $x\in E$ and for every $n\in\N$, there is an $m>n$
  such that the Kolmogorov complexity of the first $m$ digits in the binary expansion of $x$ is less
  than or equal to~$b\cdot m$.  Let such $x$ and $n$ be given, and consider a stage $s+1$ such that
  $m_{s+1}$ is greater than $n$.  Then, the first $s$ many steps of the construction together with
  the values of $a_{\ell_{s+1}},$ $b_{\ell_{s+1}}$ and $m_{s+1}$ can be effectively described by a
  sequence of length $m_{s+1}$.  The result of stage $s+1$ of the construction is effectively
  determined from these parameters; in particular, $E_{\ell_{s+1}}$ is effectively defined from
  these parameters.  By definition, there are at most $m_{\ell_{s+1}}/q_{\ell_{s+1}}$ many intervals
  in $E_{\ell_{s+1}}$.  By Lemma~\ref{6},
  \[
  \frac{1}{q_{\ell_{s+1}}}<
   m_{\ell_{s+1}}^{a_{\ell_{s+1}}(b_{\ell_{s+1}}+(1-2^{\ell_{s+1}-\ell_s})\epsilon_{s+1})-1}<
   m_{\ell_{s+1}}^{a_{\ell_{s+1}}(b_{\ell_{s+1}}+\epsilon_{s+1})-1},
    \]
    and so
    \[
     \frac{m_{\ell_{s+1}}}{q_{\ell_{s+1}}}<
     m_{\ell_{s+1}}^{a_{\ell_{s+1}}(b_{\ell_{s+1}}+\epsilon_{s+1})}=
     2^{\log(m_{\ell_{s+1}})a_{\ell_{s+1}}(b_{\ell_{s+1}}+\epsilon_{s+1})}.
     \]
    Thus, we can read off a surjection from the set of binary sequences of length
    \[\log(m_{\ell_{s+1}})a_{\ell_{s+1}}(b_{\ell_{s+1}}+\epsilon_{s+1})\] 
     to the set of intervals in
    $E_{\ell_{s+1}}$.  Each interval $I$ in $E_{\ell_{s+1}}$ has length
    $2/m_{\ell_{s+1}}^{a_{\ell_{s+1}}}$.  Computably, each such interval $I$ restricts the first
    $\log(m_{\ell_{s+1}})a_{\ell_{s+1}}$ digits in the base-2 expansions of its elements to at most
    two possibilities.  Thus, for each $x$ in $E$, the sequence of the first
    $\log(m_{\ell_{s+1}})a_{\ell_{s+1}}$ digits in its base-2 expansion can be uniformly computably
    described using the information encoded by three sequences, one of length
    $\log(m_{\ell_{s+1}}) a_{\ell_{s+1}}(b-(b_{\ell_{s+1}}+\epsilon_{s+1}))$ to describe the
    construction up to stage $s$, one of length
    $ \log(m_{\ell_{s+1}}) a_{\ell_{s+1}}(b_{\ell_{s+1}}+\epsilon_{s+1})$ to describe the interval
    within $E_{\ell_{s+1}}$ that contains $x$, and one of length $1$ to describe which of the two
    possibilities within that interval apply to $x$.  By the choice of $m_{\ell_{s+1}}$, this sum
    is less than or equal to $\log(m_{\ell_{s+1}})a_{\ell_{s+1}} b$, as required.
\end{proof}

\bibliography{eie}
\bigskip
\bigskip

\noindent
\begin{minipage}{\textwidth}
\small
Ver\'onica Becher
\\
Departamento de Computaci\'on,
Facultad de Ciencias Exactas y Naturales, Universidad de Buenos Aires.
Pabell\'on I, Ciudad Universitaria, (1428) Buenos Aires, Argentina.
Also, CONICET Argentina.
\\
Email: \texttt{vbecher@dc.uba.ar}
\medskip

Jan Reimann
\\
Department of Mathematics Penn State University. 
318B McAllister, University Park, PA 16802
\\
 Email: \texttt{jan.reimann@psu.edu}
\medskip

Theodore A. Slaman
\\
The University of California, Berkeley,
Department of Mathematics.
719 Evans Hall \#3840,
Berkeley, CA 94720-3840 USA
\\
Email: \texttt{slaman@math.berkeley.edu}
\end{minipage}

\end{document}